\documentclass[14pt]{amsart}

\usepackage{pifont}
\usepackage{mathrsfs}
\usepackage{amscd,color}
\usepackage{amsmath}
\usepackage{latexsym}
\usepackage{amsfonts}
\usepackage{amssymb}
\usepackage{amsthm}
\usepackage{graphicx}
\usepackage{amsmath,amsfonts}
\usepackage{CJK}
\usepackage[linkcolor=blue,citecolor=blue]{hyperref}

\makeatletter
\newcommand{\Extend}[5]{\ext@arrow0099{\arrowfill@#1#2#3}{#4}{#5}}
\makeatother
\let\pa=\partial

\let\r=\rho

\def\eps\epsilon

\def\cA{{\cal A}}

\def\C{\mathop{\bf C\kern 0pt}\nolimits}
\def\DD{\mathop{\bf D\kern 0pt}\nolimits}
\def\K{\mathop{\bf K\kern 0pt}\nolimits}
\def\N{\mathop{\bf N\kern 0pt}\nolimits}
\def\Q{\mathop{\bf Q\kern 0pt}\nolimits}

\newcommand{\beq}{\begin{equation}}
\newcommand{\eeq}{\end{equation}}
\newcommand{\ben}{\begin{eqnarray}}
\newcommand{\een}{\end{eqnarray}}
\newcommand{\beno}{\begin{eqnarray*}}
\newcommand{\eeno}{\end{eqnarray*}}

\def\R{\mathop{\mathbb R\kern 0pt}\nolimits}

\newtheorem{theorem}{Theorem}[section]
\newtheorem{proposition}[theorem]{Proposition}
\newtheorem{lemma}[theorem]{Lemma}
\newtheorem{remark}{Remark}[section]

\theoremstyle{remark}

\begin{document}

 \title[Scattering theory for focusing $2d$ INLS]{\bf A Remark on the Scattering theory for  the $2D$ radial focusing
 INLS}

\author{Chengbin Xu}%
\address{The Graduate School of China Academy of Engineering Physics, P. O. Box 2101, Beijing, China, 100088}
\email{xcbsph@163.com}

\author{Tengfei Zhao}%
\address{Beijing Computational Science Research Center,
No. 10 West Dongbeiwang Road, Haidian District, Beijing, China, 100193 }
\email{zhao\underline{ }tengfei@csrc.ac.cn}%

\maketitle
\vspace{-1.3cm}

\begin{abstract}
  We consider the  scattering results of the radial
  solutions below the ground state to the focusing inhomogeneous nonlinear Schr\"odinger equation
  $$i\partial_tu+\Delta u +|x|^{-b}|u|^{p}u=0$$
  in two dimension, where $0<b<1$ and $2-b<p<\infty$. We use a modified version of Arora-Dodson-Murphy's approach \cite{ADM} to give a new proof that extends the scattering results of \cite{FG-2019} and avoids concentration compactness.
\end{abstract}

\begin{center}
 \begin{minipage}{12cm}
 \end{minipage}
 \end{center}

\begin{CJK*}{GBK}{song}

\setcounter{section}{0}\setcounter{equation}{0}
\section{Introduction}

\noindent

 We consider the Cauchy problem of the inhomogeneous nonlinear Schr\"odinger equation (INLS)
 \begin{align}\label{INLS}
 \begin{cases}
   &i\pa_tu+\Delta u+|x|^{-b}|u|^{p}u=0,\ \ \ \  t\in\R,\ x\in\R^2,\\
   &u(0,x)=u_0(x) \in H^1(\R^2),
 \end{cases}
 \end{align}
with  $0<b<1$ and $2-b<p<\infty$.
The equation \eqref{INLS} arises naturally in nonlinear optics for the
propagation of laser beams,  see \cite{Dinh-blowup} \cite{Dinh-es}. 

The equation \eqref{INLS} is $\dot H^{s_p}(\R^2)$-critical in the sense that the $\dot H^{s_p}(\R^2)$ norm of initial data is invariant under the standard scaling
$$u_{\lambda}(t,x)=\lambda^{\frac{2-b}{p}}u(\lambda^2 t,\lambda x),$$
where
$s_p=\frac{d}{2}-\frac{2-b}{p}.$
The solutions to equation \eqref{INLS} conserve the mass, defined by
$$M(u):=\int_{\R^2}|u|^2dx=M(u_0),$$
and the energy, defined as the sum of the kinetic and potential energies:
$$E(u):=\int_{\R^2}\frac12|\nabla u|^2-\frac1{p+2}|x|^{-b}|u|^{p+2}dx=E(u_0).$$
From $2-b<p<\infty$,
we   have $0<s_p<1$, which implies that the equation \eqref{INLS} is mass supercritical and energy subcritical.

Now, we recall the well-posedness theory of the equation \eqref{INLS}.
 In \cite{GS}, Genoud and Stuart proved that solution to the Cauchy problem of \eqref{INLS} is locally well-posed in $H^1(\R^d)$ for $0<b<\min\{2,d\}$. More recently, Guzm$\acute{a}$n \cite{G} established the local well-posedness of \eqref{INLS} based on Strichartz estimates. More precisely,   for $d\geq4$ with $0<b<2$ or $d=1,2,3$ with  $0<b<\frac{d}3,$ the Cauchy problem \eqref{INLS} is locally well-posed in $H^1(\R^d)$. Dinh \cite{Dinh} extended Guzm$\acute{a}$n's results to larger regions. 


This equation admits a global nonscattering solution of the form $u(t)=e^{it}Q,$ where $Q$ is the ground state solution to the elliptic equation
$$-\Delta Q +Q-|x|^{-b}|Q|^{p}Q=0.$$ The existence of the ground state is proved in Genoud \cite{G},\cite{G1}, Genoud and Stuart \cite{GS}, while the uniqueness is handled in Yangida \cite{Y}, Genoud \cite{G2}.
Moreover,  the global well-posedness and scattering theory of solutions to \eqref{INLS} under the ground states $Q$ have been studied in  \cite{Campos-2019},\cite{FG-2017},\cite{FG-2019}, which extended the results of \cite{DHR-2008,HR} for nonlinear Schr\"odinger equation with $|u|^pu$ type nonlinearities.

In this paper, we consider the mass-supercritical case 
 in two spatial dimension. We will give another proof of the following scattering results.

\begin{theorem}\label{TMain}
  Let $0<b<1$ and $2-b<p<\infty$. Suppose $u_0$ is radial and such that
  $$M(u_0)^{1-s_p}E(u_0)^{s_p}<M(Q)^{1-s_p}E(Q)^{s_p},$$
and
$$ \|u_0\|_{L^2}^{1-s_p}\|\nabla u_0\|_{L^2}^{s_p}<\|Q\|_{L^2}^{1-s_p} \|\nabla Q\|_{L^2}^{s_p}.$$
  Then the solution $u$ to equation \eqref{INLS} with initial data $u_0$ is globally well-posed and scatters, that is, there exist
  $u_\pm\in H^1(\R^2)$, such that
$$ \lim_{t\to\pm\infty}\|u(t)-e^{it \Delta}u_\pm\|_{H^1(\R^2)}=0  .$$
\end{theorem}

By the concentration compactness/rigidity method, Farah-Guzm\'{a}n \cite{FG-2019} proved Theorem \ref{TMain} for the case $0<b<\frac23$.
However, the well-posedness theory of \eqref{INLS} for $0<b<1$ has been proved by \cite{Dinh}.
  Then, inspired by the new approach of \cite{ADM, DM-2017-PAMS}, we present another proof that avoids the concentration compactness. Our proof is based on a Virial/Morawetz estimate, the radial Sobolev embedding and a scattering criterion that we will establish.

The rest of this paper is organized as follows: In section 2, we set up some notations and recall some basic properties for the equation \eqref{INLS}.  We will  prove a new scattering criterion for \eqref{INLS} 
in section 3. In section 4, via the Morawetz identity, we will establish the Virial/Morawetz estimates and then show that the solution satisfies the scattering criterion of Lemma \ref{SC}, thereby completing the proof of Theorem \ref{TMain}.

We conclude the introduction by giving some notations which
will be used throughout this paper. We always use $X\lesssim Y$ to denote $X\leq CY$ for some constant $C>0$.
Similarly, $X\lesssim_{u} Y$ indicates there exists a constant $C:=C(u)$ depending on $u$ such that $X\leq C(u)Y$.
We also use the big-oh notation $\mathcal{O}$. e.g. $A=\mathcal{O}(B)$ indicates $C_{1}B\leq A\leq C_{2}B$ for some constants $C_{1},C_{2}>0$.
The derivative operator $\nabla$ refers to the spatial  variable only.
We use $L^r(\mathbb{R}^2)$ to denote the Banach space of functions $f:\mathbb{R}^2\rightarrow\mathbb{C}$ whose norm
$$\|f\|_r:=\|f\|_{L^r}=\Big(\int_{\mathbb{R}^2}|f(x)|^r dx\Big)^{\frac1r}$$
is finite, with the usual modifications when $r=\infty$. For any non-negative integer $k$,
we denote by $H^{k,r}(\mathbb{R}^2)$ the Sobolev space defined as the closure of smooth compactly supported functions in the norm $\|f\|_{H^{k,r}}=\sum_{|\alpha|\leq k}\|\frac{\partial^{\alpha}f}{\partial x^{\alpha}}\|_r$, and we denote it by $H^k$ when $r=2$.
For a time slab $I$, we use $L_t^q(I;L_x^r(\mathbb{R}^2))$ to denote the space-time norm
\begin{align*}
  \|f\|_{L_{t}^qL^r_x(I\times \R^2)}=\bigg(\int_{I}\|f(t,x)\|_{L^r_x}^q dt\bigg)^\frac{1}{q}
\end{align*}
with the usual modifications when $q$ or $r$ is infinite, sometimes we use $\|f\|_{L^q(I;L^r)}$ or $\|f\|_{L^qL^r(I\times\mathbb{R}^2)}$ for short.

\section{Preliminaries}
 We start this section by introducing some notations used throughout the paper. 
We say the pair $(q,r)$ is $L^2-$adimissible or simply admissible pair if it satisfies the condition
$\frac{2}{q}=\frac{d}{2}-\frac{d}{r}$
and
\begin{eqnarray}
\begin{cases}
  2\leq r\leq \frac{2d}{d-2},\ \ d\geq3;\\
  2\leq r<\infty,\ \ d=2;\\
  2\leq r\leq \infty; \ \ d=1.
\end{cases}
\end{eqnarray}
For $s>0$, we also say the pair $(q,r)$ is  $\dot{H^s}-$admissible if
$\frac{2}{q}=\frac{d}{2}-\frac{d}{r}-s$
and
\begin{eqnarray}
\begin{cases}
  \frac{2d}{d-2s}\leq r\leq (\frac{2d}{d-2})^-,\ \ d\geq3;\\
  \frac{2}{1-s}\leq r\leq((\frac{2}{1-s})^+)^{'},\ \ d=2;\\
  \frac2{1-2s}\leq r\leq \infty; \ \ d=1.
\end{cases}
\end{eqnarray}
Here, $a^-$ is a fixed number and slightly smaller than $a$( i. e., $a^-=a-\epsilon$, where $\epsilon$ is small enough) and,  we define $a^+$ in a similar way. Moreover, we denote $(a^+)^{'}$ is the number such that
$$\frac{1}{a}=\frac{1}{(a^+)^{'}}+\frac{1}{a^+},$$
that is, $(a^+)^{'}=\frac{a \, a^+}{a^+-a}. $
 Finally, we say that $(q,r)$ is $\dot{H^{-s}}-$admissible if
$\frac{2}{q}=\frac{d}{2}-\frac{d}{r}+s$
and
\begin{eqnarray}
\begin{cases}
  (\frac{2d}{d-2s})^+\leq r\leq (\frac{2d}{d-2})^-,\ \ d\geq3.\\
  (\frac{2}{1-s})^+\leq r\leq((\frac{2}{1+s})^+)^{'},\ \ d=2;\\
  (\frac2{1-2s})^+\leq r\leq \infty; \ \ d=1.
\end{cases}
\end{eqnarray}
Given $s\in\R,$ let $\Lambda_s=\{(q,r):(q,r) \ \text{is}\  \dot{H^s}-\text{admissible} \}$. We define the following Strichartz norm
$$\|u\|_{S(\dot{H}^s,I)}=\sup_{(q,r)\in \Lambda_s}\|u\|_{L_t^qL_x^r(I\times\R^2)}$$
and dual Strichartz norm
$$\|u\|_{S^{'}(\dot{H}^{-s},I)}=\inf_{(q,r)\in \Lambda_{-s}}\|u\|_{L_t^{q^{'}}L_x^{r^{'}}(I\times\R^2)}.$$
If $s=0,$ we shall write $S^{'}(\dot{H}^{0},I)=S^{'}(L^2,I)$, $S(\dot{H}^0,I)=S(L^2,I)$. If $I=\R$, we will often omit $I$.
Now, we recall  the radial Sobolev embedding and the Strichartz estimates.
\begin{lemma}[Radial Sobolev embedding\cite{ADM} \cite{Tao-2004}]\label{lem2.1}
For radial $f\in H^1(\R^2)$, then
  \begin{align*}
    \||x|^{\frac{1}{2}}f\|_{L^\infty(\R^2)}\lesssim ||f||_{H^1(\R^2)}.
  \end{align*}
%

\end{lemma}
\begin{lemma}[Strichartz estimates\cite{KeelTao}]
  The following estimates hold.

 $(i)$(linear estimate).
  $$\|e^{it\Delta}f\|_{S(\dot{H}^s)}\leq C\|f\|_{\dot{H}^s}.$$

  $(ii)$(inhomogeneous  estimates).
  $$\left\|\int_{0}^{t}e^{i(t-s)\Delta}g(\cdot,s)ds\right\|_{S(\dot{H}^s,I)}\leq C\|g\|_{S^{'}(\dot{H}^{-s},I)}.$$
\end{lemma}

Next, we recall some interpolation estimates for the nonlinear term.
\begin{lemma}\label{inh-str}
  Let $2-b<p<\infty$ and $0<b<2$. If $s_p<1$, then for $j \in \{1,2\},$ there exist  $\theta_j\in (0,p)$  sufficiently small so that the following hold true:
   \begin{equation}\label{Str-1}
     \||x|^{-b}|u|^pv\|_{S^{'}(\dot{H}^{-s_p},I)}\leq C\|u\|_{L_t^\infty H_x^1(I\times\R^2)}^{\theta_1}
     \|u\|_{S(\dot{H}^{s_p},I)}^{p-\theta_1}
     \|v\|_{S(\dot{H}^{s_p},I)}
   ,\end{equation}
  \begin{equation}
    \label{Str-2}\||x|^{-b}|u|^pv\|_{S^{'}(L^2,I)}\leq C\|u\|_{L_t^\infty H_x^1(I\times\R^2)}^{\theta_2}\|u\|_{S(\dot{H}^{s_p},I)}^{p-\theta_2}
    \|v\|_{S(L^2,I)}
   ,\end{equation}
   where $C>0$ is some constant independent of $u$.
\end{lemma}
 \begin{proof}
 For a sufficiently small number $\theta>0$,
we define the following numbers (depending on $p$ and $b$)
 $$\hat{q}=\frac{2p(p+2-\theta)}{p(p+b)-\theta(p-2+b)},\quad \hat{r}=\frac{2p(p+2-\theta)}{(p-\theta)(2-b)},$$
 and
 $$\tilde{a}=\frac{p(p+2-\theta)}{p(p+b-\theta)-(2-b)(1-\theta)},\quad \hat{a}=\frac{p(p+2-\theta)}{2-b}.$$
     It is easy to see that $(\hat{q},\hat{r})$ is $L^2$-admissible, $(\hat{a},\hat{r})$ is $\dot{H}^{s_p}$-admissible and $(\tilde{a},\hat{r})$ is $\dot{H}^{-s_p}$-admissible. Moreover, we observe that
$$\frac1{\hat{a}}+\frac1{\tilde{a}}=\frac2{\hat{q}}.$$
By H\"older and Sobolev 
(see \cite{Gz} for details), we have
\begin{align}\label{Str-4}
\||x|^{-b}|u|^{p}v\|_{L_x^{\hat{r}'}}\lesssim \|u\|_{H^1}^{\theta}\|u\|_{L_x^{\hat{r}}}^{p-\theta}\|v\|_{L_x^{\hat{r}}},
\end{align}
so that \eqref{Str-1} and \eqref{Str-2} follow.
\end{proof}
\begin{lemma}[\cite{Dinh}, Lemma $6.4$]\label{inh-str2}
 Let $b\in(0,1)$ and $2-b<p<\infty$. Then there exist $(p_1,q_1),\ (p_1, q_2)\in\Lambda_0$ satisfying
 $2p+2>p_1,p_2$,
  and
  $$\|\nabla(|x|^{-b}|u|^{p}u)\|_{S'(L^2,I)}\lesssim(\|u\|_{L_t^{m_1}L_x^{q_1}(I\times\R^2)}^p
  +\|u\|_{L_t^{m_2}L_x^{q_2}(I\times\R^2)}^p)\|\langle\nabla \rangle u\|_{S(L^2,I)},$$
  where $m_1=\frac{\alpha p_1}{p_1-2}$ and $m_2=\frac{\alpha p_2}{p_2-2}.$
\end{lemma}
Denote $C_{p,d,b}=\frac{\frac{dp}{2}+b}{p+2-(\frac{dp}{2}+b)}$, then we have the following Gagliardo-Nirenberg inequality.

\begin{proposition}[\cite{F}]
  Let $\frac{4-2b}{d}<p<\frac{4-2b}{d-2}$ and $0<b<\min\{2,d\},$ then we have 
  \begin{align}\label{G-N}
  \int_{\R^d}|x|^{-b}|u|^{p+2}dx\leq C_0\|\nabla u\|_{L^2}^{\frac{dp}{2}+b}\|u\|_{L^2}^{p+2-(\frac{dp}{2}+b)}.
  \end{align}
  Here  the sharp constant $C_0>0$ is explicitly given by
  $$C_0=C_{p,d,b}^{\frac{4-pd-2b}{4}}\frac{p+2}{(\frac{dp}{2}+b)
  \|Q\|_{L^2}^{p}},$$
  where $Q$ is the unique non-negative, radially-symmetric, decreasing solution of the equation
 $$\Delta Q-Q+|x|^{-b}|Q|^{p}Q=0.$$
Moreover the solution $Q$ satisfies the following relations
 $$\|\nabla Q\|_{L^2}=C_{p,d,b}^{\frac12}\|Q\|_{L^2},$$
  and
  $$\int_{\R^d}|x|^{-b}|Q|^{p+2}dx=C_{p,d,b}\|Q\|_{L^2}^2.$$
\end{proposition}

Now, we recall the well-posedness theory of equation in $H^1(\R^2)$ under the ground state $Q$, which is proved by  Farah \cite{F} and  Dinh \cite{Dinh}.

\begin{theorem}[ ]

  Let $d=2,\ 0<b<1,\  2-b<p<\infty$, and $s_p=1-\frac{2-b}{p}.$
Suppose that $u(t)$ is the solution of \eqref{INLS} with initial data $u_0\in H^1$ satisfying
  $$E(u_0)^{s_p}M(u_0)^{1-s_p}<E(Q)^{s_p}M(Q)^{1-s_p},$$
  and
  $$\|\nabla u_0\|_{L^2}^{s_p}\|u_0\|_{L^2}^{1-s_p}<\|Q\|_{L^2}^{s_p}\|Q\|_{L^2}^{1-s_p}.$$
  Then $u(t)$ is   globally well-posed in $H^1$.
Moreover, the solution satisfy $u \in L_{loc}^q(\R, L^r(\R^2))$ for any Schr\"odinger admissible pair $(q,r)$ and
   $$\|\nabla u(t)\|_{L^2}^{s_p}\|u(t)\|_{L^2}^{1-s_p}<\|\nabla Q\|_{L^2}^{s_p}\|Q\|_{L^2}^{1-s_p}.$$
\end{theorem}


\section{Scattering criterion}\label{s-c-sec}

\noindent
In this section, we prove a new scattering criterion for the solutions of the equation \eqref{INLS}.

\begin{lemma}
\label{SC}
Let $0<b<1$ and $2-b<p<\infty$. Suppose $u:\R_t\times\R^2\rightarrow \mathbb{C}$ is a radial solution to \eqref{INLS}
and  such that
  \begin{align}\label{zz0}
  M(u_0)=E(u_0)=E\ \  \text{and}\ \
    \|u\|_{L_t^\infty H_x^1(\R\times\R^2)}\leq E,
  \end{align}
  for some $0<E<\infty.$
There exist $\epsilon=\epsilon(E)>0$ and $R=R(E)>0$ such that if
  \begin{align}\label{zz1}
   \liminf_{t\to \infty}\int_{|x|\leq R}|u(t,x)|^{2}dx\leq \epsilon ^{2},
  \end{align}
  and 
  \begin{align}\label{zz2}
  \int_{0}^{T}\int_{\R^2}|x|^{-b}|u(t,x)|^{p+2}dx\leq T^{\alpha},
  \end{align}
for some  $0<\alpha<1$,
then $u$ scatters forward in time.
\begin{remark}\label{remark-scatter-criter}
  In fact, the inequality \eqref{zz1} follows from \eqref{zz2}. Indeed, for fixed  $R\gg1$, we may have
  $$R^{-b}\int_{0}^{T}\int_{|x|\leq R}|u|^{p+2}dxdt\leq T^{\alpha}.$$
  By the mean value theorem of calculus on interval $[0,T]$, there exists time sequence $\{t_n\}\rightarrow \infty$ such that
  $$\int_{|x|\leq R}|u(t_n)|^{p+2}dx\rightarrow0,\ \ n\rightarrow\infty.$$
 Then, the inequality  \eqref{zz1} follows from this estimate and the fact that $\|u\|_{L_x^2(|x|\leq R)}\leq R^{1-\frac2{p+2}}\|u\|_{L_x^{p+2}(|x|\leq R)}.$  
\end{remark}

\end{lemma}
\begin{proof}

 First, we claim that it suffices to show
 \begin{equation}\label{u-S-bdd}
 \|u\|_{S(\dot{H}^{s_p})}<\infty.
 \end{equation}
 In fact, by a standard continuity argument and Lemma \ref{inh-str2}, the conclusion follows if one has
 $$\|u\|_{L_t^{m_1}L_x^{q_1}},\ \|u\|_{L_t^{m_2}L_x^{q_2}}<\infty,$$
 where $(m_1,q_1)$ and $(m_2,q_2)$ are given in Lemma \ref{inh-str2}.
 However, these two estimates are the consequences of the interpolation between
 \eqref{u-S-bdd} and the assumption \eqref{zz0}.

%
Next, we will show the estimate \eqref{u-S-bdd}.
By Lemma \ref{inh-str}, Duhamel formula and continuity argument, we need to show
  $$\|e^{i(t-T_0)\Delta}u(T_0)\|_{S(\dot{H}^{s_p},[T_0,\infty))}\ll 1.$$
We rewrite $e^{i(t-T_{0})\Delta}u(T_{0})$ as
\begin{align*}
  e^{i(t-T_{0})\Delta}u(T_{0})=e^{it\Delta}u_0-iF_{1}(t)-iF_{2}(t),
\end{align*}
where
\begin{align*}
  F_{j}(t):=\int_{I_{j}}e^{i(t-s)\Delta}(|x|^{-b}|u|^{p}u)(s)ds, ~j=1,2,
\end{align*}
with $I_{1}=[0,T_{0}-\epsilon ^{-\theta}]$ and  $I_{2}=[T_{0}-\epsilon ^{-\theta},T_{0}]$. Here, $\theta>0$ will be chosen later.
Let $T_0$ be large enough, then we have
$$
  \|e^{it\Delta}u_0\|_{S(\dot{H}^{s_p},[T_0,\infty))} \ll1.
$$
Hence, it remains to show
\begin{align*}
  \|F_{j}(t)\|_{S(\dot{H}^{s_p},[T_0,\infty))} \ll 1,\quad \text{~ for ~} j=1,2.
\end{align*}
\textbf{Estimation  of $F_{1}(t)$}: We may use the dispersive estimate, H\"older's inequality, and the assumption \eqref{zz2} to obtain
\begin{align*}
&\left\|\int_{0}^{T_0-\epsilon^{-\theta}}e^{i(t-s)\Delta}|x|^{-b}|u|^puds\right\|_{L_x^\infty}\\
\lesssim~&
\int_{0}^{T_0-\epsilon^{-\theta}}|t-s|^{-1}\||x|^{-b}|u|^pu\|_{L_x^1}ds\\
\lesssim~ & \int_{0}^{T_0-\epsilon^{-\theta}}|t-s|^{-1}\left(\||x|^{\frac{-b}{p+2}}u\|_{L_x^{p+2}}^{\frac{p+2}{2}}\||x|^{\frac{-b}{2}}\|
_{L_x^{\frac{4}{b^+}}(B_1)} \||u|^\frac{p}{2}\|_{\frac{4}{2-p-}}\right)ds\\
&+\int_{0}^{T_0-\epsilon^{-\theta}}|t-s|^{-1}\left(\||x|^{\frac{-b}{p+2}}u\|_{L_x^{p+2}}^{\frac{p+2}{2}}\|
\||x|^{\frac{-b}{2}}\|_{L_x^{\frac{4}{b^-}}(B_1^c)} \||u|^\frac{p}{2}\|_{\frac{4}{2-p+}} \right)ds\\
\lesssim~ & \int_{0}^{T_0-\epsilon^{-\theta}}|t-s|^{-1}\||x|^{\frac{-b}{p+2}}u\|_{L_x^{p+2}}^{\frac{p+2}2}ds\\
\lesssim ~ & T_0^{\frac{\alpha}{2}}\epsilon^{\frac{\theta}{2}},
\end{align*}
which yields
$$\left\|\int_{0}^{T_0-\epsilon^{-\theta}}e^{i(t-s)\Delta}|x|^{-b}|u|^puds\right\|_{L_{t,x}^\infty
(T_0,\infty)\times\R^2}\lesssim T_0^{\frac{\alpha}{2}}\epsilon^{\frac{\theta}{2}}.$$
On the other hand, we may rewrite $ F_1$ as
\begin{equation}\label{F1}
F_1(t)=e^{i(t-T_0+\epsilon^{-\theta})\Delta}u(T_0-\epsilon^{-\theta})-e^{it\Delta}u_0.
\end{equation}
For any $(q,r)\in \Lambda_{s_p}$, we define the pair $(q_0,r_0)$   by
$$q_0=(1-s_p)q,\ \ \ r_0=(1-s_p)r,$$
then $(q_0,r_0)\in\Lambda_0.$
By Strichartz  and the equality \eqref{F1},
 we have
$$\|F_1(t)\|_{L_t^{q_0}L_x^{r_0}}\lesssim 1.$$
Thus, by interpolation, we get
$$\|F_1(t)\|_{S(\dot{H}^{s_p},[T_0,\infty))}\lesssim (T_0^{\alpha}\epsilon^{\theta})^{\frac{1-s_p}{2}}.$$
\textbf{Estimation  of $F_{2}(t)$}: By the inequality \eqref{Str-1}, Sobolev  and radial Sobolev embedding, we get
\begin{equation}\label{F-2-1}
\begin{split}
 \|F_2(t)\|_{S(\dot{H}^{s_p},[T_0,\infty))}\lesssim & \||x|^{-b}|u|^pu\|_{S^{'}(\dot{H}^{-s_p})}\\
 \lesssim~&\|u\|_{L_t^\infty H_x^1}^{\eta}\|u\|_{S(\dot{H}^{s_p})}^{p+1-\eta}\\
 \lesssim~&\sup_{(q,r)\in\Lambda_{s_p}}\|u\|_{L_t^qL_x^r(I_2)}^{p+1-\eta}\\
 \lesssim~&\sup_{(q,r)\in\Lambda_{s_p}}\|u\|_{L_t^{\infty}L_x^r(I_2)}^{p+1-\eta}\epsilon^{\frac{-(p+1-\eta)\theta}{q}}.
\end{split}
\end{equation}
 By the assumption \eqref{zz1}, we may choose $T>T_0$ so that
\begin{equation}\label{3.10-cor}
\int\chi_R(x)|u(T,x)|^2<\epsilon^2.
\end{equation}
 Here $\chi_R(x):=\chi(\frac xR)$ for $R>0$, where  $\chi(x)$ is a  radial smooth function such that
 \begin{equation*}
  \chi(x)=\left\{
    \begin{aligned}
    &1,\ \  |x|\leq \tfrac12,\\
    &0,\ \  |x|>1.
 \end{aligned}\right.
 \end{equation*}
 Using the identity
$$\pa_t|u|^2=-2\nabla\cdot \text{Im}(\bar{u}\nabla u)$$
 together with \eqref{3.10-cor}, integration by parts, and Cauchy-Schwarz, we can deduce
$$\left|\pa_t\int_{I_2}\chi_R|u|^2ds\right|\lesssim \frac{1}{R}.$$
Thus, for $R\gg \epsilon^{-(2+\theta)}$, we find
$$\|\chi_Ru\|_{L_t^\infty L_x^2(I_2\times \R^2)}\lesssim \epsilon.$$
Using the radial Sobolev inequality and choosing $R$ large enough, we have
\begin{equation*}
     \|u\|_{L_t^\infty L_x^r  }
\lesssim\|\chi_Ru\|_{L_t^\infty L_x^2 }^\frac{1}{r }
\|  u\|_{L_t^\infty L_x^{2(r-1)} }^{\frac{r-1}{r}}
   +\|\left(1-\chi_R\right)u\|_{L_{t,x}^\infty  }^\frac{r-2}{r}
\|  u\|_{L_t^\infty L_x^{2} }^\frac{2}{r}
\lesssim~  \epsilon^{\frac1{r}},
\end{equation*}
where the time-space norms are over the region $I_2\times \R^2$.

By the definition of $\Lambda_{s_p}$, there exists $\delta,c>0$ such that $\frac{2}{1-s_p}+\delta \leq r\leq c.$ Thus, by \eqref{F-2-1}, we have
\begin{align*}
\|F_2(t)\|_{S(\dot{H}^{s_p},[T_0,\infty))}\lesssim~& \sup_{r}\|u\|_{L_t^\infty L_x^r(I_2)}^{p+1-\eta}\epsilon^{\frac{-(p+1-\eta)\theta}{2}}\\
\lesssim~& \epsilon^{(p+1-\eta)(\frac{1}{c}-\frac{\theta}{2})}.
\end{align*}
Choosing $\theta=\frac1c,\ \gamma $ such that $\epsilon=T^{-c(1+\gamma)}$ and $\gamma+\alpha<1$, we get
 $$\|e^{i(t-T_0)\Delta}u(T_0)\|_{S(\dot{H}^{s_p},[T_0,\infty))}\ll 1.$$
The proof is completed.
\end{proof}

\section{Proof of Theorem \ref{TMain}}

\noindent

Throughout this section, we suppose $u(t)$ is a solution to equation \eqref{INLS} satisfying the hypotheses of Theorem \ref{TMain}. In particular, we have that $u$ is global and uniformly bounded in $H^1$.
 Furthermore, we will see that  there exists $\delta > 0$ so that
\begin{align}\label{b}
  \sup_{t\in \R}\|u(t)\|_{L^2}^{1-s_p}\|\nabla u(t)\|_{L^2}^{s_p}~<~(1-2\delta)\|Q\|_{L^2}^{1-s_p}\|\nabla Q\|_{L^2}^{s_p}.
\end{align}

We will prove the following Morawetz estimates. 

\begin{proposition}[Morawetz estimates]\label{pro1}
  Let $d=2$, $0<b<2,\ \ 2-b<p<\infty$ and $u$ be a solution to the focusing equation
\eqref{INLS} on the space-time slab $[0,T]\times \R^d$. Then there exists constant $0<\beta<1$ such that
  \begin{align}\label{Main1}
  \int_{0}^{T}\int_{\R^d}|x|^{-b}|u(t,x)|^{p+2}dxdt<T^{\beta}.
  \end{align}
\end{proposition}

Using Proposition \ref{pro1}, Remark \ref{remark-scatter-criter}, rescaling, and the scattering criterion in Section \ref{s-c-sec}, we can quickly prove Theorem \ref{TMain}. 

Next, we prove Proposition \ref{pro1} by a  Morawetz identity.
First, we recall the necessary coercivity property, that is, the inequality \eqref{b} holds on large balls.

 \begin{lemma}[Coercivity on balls] \label{Cor}
There exists $R=R(\delta, M(u), Q)>0$ sufficiently large so that
$$\sup_{t\in\R}\|\chi_{R}u\|_{L_x^2}^{1-s_p}\|\chi_{R}u\|_{\dot{H}^1}^{s_p}<
(1-\delta)\|Q\|_{L_x^2}^{1-s_p}\|Q\|_{\dot{H}^1}^{s_p}.$$
 In particular, there exists $\delta'$ so that
 $$\int|\nabla(\chi_{R}u)|^2dx-\frac{b+p}{p+2}\int|x|^{-b}|\chi_Ru|^{p+2}dx
 \geq\delta'\int|x|^{-b}|u|^{p+2}.$$
\end{lemma}
\begin{proof}
The proof follows from the conservations of mass and energy and the
Gagliardo-Nirenberg inequality \eqref{G-N}.
We refer to \cite{ADM} for an analogous proof.
\end{proof}

By a direct computation, we have the following Virial/Morawetz identity.
\begin{lemma}[Virial/Morawetz identity]\label{Mor-Iden}
  Let $u$ be the solution of \eqref{INLS} and $a(x)$ be a smooth function. We denote the Morawetz action $M_a(t)$ by
  $$M_a(t)=2\int_{\R^2} \nabla a(x) \text{~Im}(\bar{u}\nabla u)(x)dx.$$
  Then we have
  \begin{align}
    \frac{d}{dt}M_a(t)~=~&-\int\Delta^2a(x)|u|^2dx+4\int \pa_{jk}a(x)Re(\pa_ku\pa_j\bar{u})dx\\
    &-\frac{2p}{p+2}\int\Delta a(x)|x|^{-b}|u|^{p+2}dx+\frac{4}{p+2}\int\nabla a(x)\cdot \nabla(|x|^{-b})|u|^{p+2}dx,
  \end{align}
where the repeated indices are summed.
\end{lemma}

Using this identity, we may prove proposition \ref{pro1} as following:

\begin{proof}

Let $R\gg 1$ to be chosen later. We take $a(x)$ to be a radial, smooth function satisfying
\begin{eqnarray}
a(x)=
\begin{cases}
|x|^2,& \text{~for~ }|x|\leq R;\\
3R|x|,& \text{~for~ } |x|>2R,
\end{cases}
\end{eqnarray}
and when $R<|x|\leq 2R$, for any multiindex $\gamma$ there hold
\begin{align*}
  \partial_{r}a\geq 0,\,\,\,\partial_{rr}a\geq 0
  \quad \text{and} \quad |\partial^{\gamma}a| \lesssim R|x|^{-|\gamma|+1}.
\end{align*}
Here $\partial_{r}$ denotes the radial derivative. Under these conditions, the matrix $(a_{jk})$ is non-negative.
And, it is easy to verify that
\begin{eqnarray*}
\begin{cases}
a_{jk}=2\delta_{jk},\quad \Delta a=4,\quad \Delta \Delta a=0,& \text{~for~ } |x|\leq R,\\
a_{jk}=\frac{3R}{|x|}[\delta_{jk}-\frac{x_{j}x_{k}}{|x|^2}],\quad \Delta a=\frac{3R}{|x|},\quad \Delta \Delta a=\frac{3R}{|x|^3},& \text{~for~ } |x|>2R.
\end{cases}
\end{eqnarray*}
Thus, we can divide $\frac{dM_a(t)}{dt}$ as follows:
\begin{equation}
\begin{split}
\frac{dM_a(t)}{dt}=&8\int_{|x|\leq R} |\nabla u|^2-\frac{p+b}{p+2}|x|^{-b}|u|^{p+2}dx\\ 
&+\int_{|x|>2R}\frac{-6pR}{|x|^{b+1}}|u|^{p+2}dx+\int \frac{-12Rb}{(p+2)|x|^{b+1}}|u|^{p+2}dx\\
&+\int_{|x|>2R}\frac{-3R}{|x|^3}|u|^2dx+\int_{|x|>2R}\frac{12R}{|x|}|\not\!\nabla u|^2 dx\\
&+\int_{R<|x|\leq 2R}4Re\bar{u}_{i}a_{ij}u_{j}+\mathcal{O}(\frac{R}{|x|^{b+1}}|u|^{p+2}+\frac{R}{|x|^3}|u|^2)dx,
\end{split}
\end{equation}
where $\not\!\nabla$ denotes the angular derivative and
subscripts denote partial derivatives.
This implies
 \begin{align}\label{Mora:e}
 \int_{|x|<R}|\nabla u|^2dx-\frac{p+b}{p+2} \int_{|x|\leq R}|x|^{-b}|u|^{p+2}dx \lesssim& \frac{dM_a(t)}{dt}+\frac{1}{R^{\alpha}},
\end{align}
where $\alpha:=\min\{2,b+\frac{p}2\}.$ From the   identity
\begin{align}
  \int\chi_R^2|\nabla u|^2=\int|\nabla(\chi_Ru)|^2)+\chi_R\Delta(\chi_R)|u|^2dx,
\end{align}
 we have
$$\|\chi_Ru\|_{\dot{H}^1}^2\lesssim \|u\|_{\dot{H}^1}^2+\frac{1}{R^2}.$$
Thus, by the radial Sobolev inequality, the Sobolev embedding and Lemma \ref{Cor}, we get
\begin{align}\label{Mora:e}
  \int_{|x|\leq \frac{R}{2}}|x|^{-b}|u|^{p+2}dx \lesssim& \frac{dM_a(t)}{dt}+\frac{1}{R^{\alpha}},
\end{align}
From this inequality  and the radial Sobolev inequality, we   have
  $$\int_{\R^2} |x|^{-b}|u|^{p+2}dx \lesssim\frac{dM_a(t)}{dt}+\frac{1}{R^{\alpha}}.$$
 Note that from the uniform $H^1$-bounds for $u$, and the choice of the  weight function $a$, we have
 $$\sup_{t\in\R}|M_a(t)|\lesssim R.$$
Then, we can use  the fundamental theorem of calculus on an interval $[0,T]$ to obtain
 $$\int_{0}^{T}\int_{\R^2} |x|^{-b}|u|^{p+2}dx  dt \lesssim R+\frac{T}{R^{\alpha}}.$$
The conclusion follows if we take $R=T^{\frac{1}{1+\alpha}}$.
\end{proof}



 \noindent\textbf{Acknowledgements}.
Tengfei Zhao is supported in part by the Chinese Postdoc Foundation Grant 2019M650457.
This work is financially supported by National Natural Science Foundation of China (NSFC-U1930402).  The authors would like to thank the referees for their helpful comments
and suggestions.


\end{CJK*}
 \end{document}